\documentclass{amsart}

\newtheorem{theorem}{Theorem}
\newtheorem{lemma}[theorem]{Lemma}

\theoremstyle{definition}

\newtheorem{definition}[theorem]{Definition}
\newtheorem{example}[theorem]{Example}

\newtheorem{remark}[theorem]{Remark}

\newcommand{\C}{{\mathbb C}}

\begin{document}

\title[Generators of von~Neumann algebras]{Generators of von~Neumann algebras associated \\ with spectral measures}

\author{A.G.~Smirnov}
\address{I.~E.~Tamm Theory Department, P.~N.~Lebedev
Physical Institute, Leninsky prospect 53, Moscow 119991, Russia}

\email{smirnov@lpi.ru}

\begin{abstract}
Let $\mathcal P_E$ be the set of all values of a spectral measure $E$ and $\mathcal A(\mathcal P_E)$ be the smallest von~Neumann algebra containing $\mathcal P_E$. We give a simple description of all sets of generators of $\mathcal A(\mathcal P_E)$ in terms of the integrals with respect to $E$. The treatment covers not only the case of generators belonging to $\mathcal A(\mathcal P_E)$, but also the case of (possibly unbounded) generators affiliated with this algebra.
\end{abstract}

\maketitle

Let $S$ be a measurable space with a $\sigma$-algebra $\Sigma$ of measurable sets, $\mathfrak H$ be a Hilbert space, and $L(\mathfrak H)$ be the algebra of all bounded everywhere defined linear operators in $\mathfrak H$. Recall~\cite{BirmanSolomjak} that a map $E$ from $\Sigma$ to the set of orthogonal projections on $\mathfrak H$ is called a spectral measure for $(S,\mathfrak H)$ if it is countably additive with respect to the strong operator topology on $L(\mathfrak H)$ and $E(S)$ is the identity operator in $\mathfrak H$. Let $\mathcal P_E$ denote the set of all projections $E(A)$ with $A\in \Sigma$, and let $\mathcal A(\mathcal P_E)$ be the von~Neumann algebra generated by $\mathcal P_E$, i.e., the smallest von~Neumann algebra containing $\mathcal P_E$. In this paper, we give a simple description of all sets of generators of $\mathcal A(\mathcal P_E)$ in terms of the integrals with respect to $E$.

This problem arises naturally when considering self-adjoint extensions of Schr\"o\-din\-ger operators with singular potentials commuting with a given set of symmetries~\cite{GSVT}. In such quantum-mechanical applications, it is convenient to consider not only generators belonging to the considered algebra, but also more general operators which are affiliated with it and may be unbounded. In this paper, we shall work in such a more general setting.

We say that an operator $T$ in $\mathfrak H$ with the domain $D_{T}$ commutes with $R\in L(\mathfrak H)$ if $R\Psi\in D_{T}$ and $RT\Psi=TR \Psi$ for any $\Psi\in D_{T}$. Given a set $\mathcal X$ of closed densely defined operators in $\mathfrak H$, let $\mathcal X'$ denote its commutant, i.e., the subalgebra of $L(\mathfrak H)$ consisting of all operators commuting with every element of $\mathcal X$. We denote by $\mathcal X^*$ the set consisting of the adjoints of the elements of $\mathcal X$. The set $\mathcal X$ is called involutive if $\mathcal X^*=\mathcal X$. If a densely defined operator $T$ commutes with $R\in L(\mathfrak H)$, then its adjoint $T^*$ commutes with $R^*$. Indeed, if $\Psi\in D_{T^*}$ and $\Phi=T^*\Psi$, then we have $\langle T\Psi',\Psi\rangle=\langle \Psi',\Phi\rangle$ for any $\Psi'\in D_T$ and, hence
\[
\langle T\Psi',R^*\Psi\rangle=\langle TR\Psi',\Psi \rangle=\langle R\Psi',\Phi \rangle = \langle \Psi',R^*\Phi \rangle,\quad \Psi'\in D_T,
\]
where $\langle \cdot,\cdot\rangle$ is the scalar product on $\mathfrak H$.
This means that $R^*\Psi\in D_{T^*}$ and $T^* R^* \Psi=R^*T^*\Psi$, i.e., $T^*$ commutes with $R^*$. It follows that
\begin{equation}\label{adjcomm}
(\mathcal X')^*=(\mathcal X^*)'.
\end{equation}

Recall~\cite{Dixmier} that a subalgebra $\mathcal M$ of $L(\mathfrak H)$ is
called a von~Neumann algebra if it is involutive and coincides with its bicommutant~$\mathcal M''$. By the well-known von~Neumann's bicommutant theorem (see,
e.~g.,~\cite{Dixmier}, Sec.~I.3.4, Corollaire~2), an involutive subalgebra $\mathcal M$ of $L(\mathfrak H)$ is a von~Neumann algebra if and only if
it contains the identity operator and is closed in the strong operator topology. It follows from~(\ref{adjcomm}) that $\mathcal X'$ is an involutive subalgebra of $L(\mathfrak H)$ for any involutive set $\mathcal X$ of closed densely defined operators in $\mathfrak H$. Moreover, as shown by the next lemma, $\mathcal X'$ is always strongly closed and, therefore, is a von~Neumann algebra for involutive $\mathcal X$ by the bicommutant theorem.

\begin{lemma} \label{l2}
Let $\mathcal X$ be a set of densely defined closed operators in $\mathfrak H$. Then $\mathcal X'$ is a strongly closed subset of $L(\mathfrak H)$.
\end{lemma}
\begin{proof}
Given an operator $T$ in $\mathfrak H$, let $C_T$ denote the set of all elements of $L(\mathfrak H)$ commuting with $T$. Since $\mathcal X'=\bigcap_{T\in\mathcal X} C_T$, it suffices to prove that $C_T$ is strongly closed for any
closed $T$. Let $R$ belong to the strong closure of $C_T$. For every $\Psi_1,\Psi_2\in\mathfrak H$ and $n=1,2,\ldots$, the set
\[
W_{\Psi_1,\Psi_2,n}=\{\tilde R\in L(\mathfrak H) : \|(\tilde R-R)\Psi_i\|<1/n,\,\,i=1,2\}
\]
is a strong neighborhood of $R$ and, hence, has a nonempty intersection with $C_T$. Fix $\Psi\in D_T$ and choose $R_n\in C_T\cap
W_{\Psi,\,T\Psi,\,n}$ for each $n$. Then $R_n \Psi\to R\Psi$ and $R_n T \Psi\to RT\Psi$ in $\mathfrak H$. As $R_n$ commute with $T$, we have $R_n
\Psi\in D_T$ and $R_nT \Psi= TR_n \Psi$ for all $n$. In view of the closedness of $T$, it follows that $R\Psi\in D_T$ and $TR\Psi = RT\Psi$, i.e.,
$R\in C_T$. The lemma is proved.
\end{proof}

A closed densely defined operator $T$ in $\mathfrak H$ is called affiliated~\cite{KR1} with a von~Neumann algebra $\mathcal M$ if $T$ commutes with every element of $\mathcal M'$. If $\mathcal X$ is a set of closed densely defined operators in $\mathfrak H$, then every element of $\mathcal X$ is obviously affiliated with the algebra $\mathcal A(\mathcal X)=(\mathcal X\cup\mathcal X^*)''$. In fact, the latter is the smallest von~Neumann algebra with this property. The algebra $\mathcal A(\mathcal X)$ will be called the von~Neumann algebra generated by $\mathcal X$, and $\mathcal X$ will be referred to as a set of generators of $\mathcal A(\mathcal X)$. If $\mathcal X\subset L(\mathfrak H)$, then $\mathcal A(\mathcal X)$ is just the smallest von~Neumann algebra containing $\mathcal X$. Clearly, a closed densely defined operator $T$ is affiliated with a von~Neumann algebra $\mathcal M$ if and only if $\mathcal A(T)\subset \mathcal M$ (here and subsequently, we write $\mathcal A(T)$ instead of $\mathcal A(\{T\})$, where $\{T\}$ is the one-point set containing $T$).

Let $E$ be a spectral measure for $(\mathfrak H,S)$. Then
$E(A_1\cap A_2)=E(A_1)E(A_2)$ for any measurable $A_1,A_2\subset S$ (see~\cite{BirmanSolomjak}, Sec.~5.1, Theorem~1) and, hence, $E(A_1)$ and $E(A_2)$ commute. This implies that the algebra $\mathcal A(\mathcal P_E)$ is Abelian. For any $\Psi\in \mathfrak H$, the finite positive measure $E_\Psi$ on $S$ is defined by setting $E_\Psi(A)=\langle E(A)\Psi,\Psi\rangle$ for any measurable $A$.
Given an $E$-measurable\footnote{As usual, a complex function $f$ defined on $S$ is called measurable if the set $f^{-1}(A)$ is measurable for any Borel subset $A$ of $\C$. A complex function is called $E$-measurable on $S$ if it is defined $E$-almost everywhere ($E$-a.e.) and coincides $E$-a.e. with a measurable function.} complex function $f$ on $S$, the integral $J^E_f$ of $f$ with
respect to $E$ is defined as the unique linear operator in $\mathfrak H$ such that
\begin{align}
&D_{J^E_f}=\left\{\Psi\in \mathfrak H : \int |f(s)|^2\,dE_\Psi(s)<\infty \right\} \label{dom} \\
& \langle\Psi, J^E_f \Psi\rangle = \int f(s)\,dE_\Psi(s),\quad \Psi\in D_{J^E_f}. \nonumber
\end{align}

For every normal\footnote{Recall that a closed densely defined linear operator $T$
in a Hilbert space $\mathfrak H$ is called normal if the operators $TT^*$ and $T^*T$ have the same domain of definition and coincide thereon.} operator $T$, there is a unique spectral measure $\mathcal E_T$ on $\C$ such that $J^{\mathcal E_T}_{\mathrm{id}_{\C}}=T$, where
$\mathrm{id}_{\C}$ is the identity function on $\C$. The operators $\mathcal E_T(A)$, where $A$ is a Borel subset of $\C$, are called the spectral
projections of $T$. If $f$ is an $\mathcal E_T$-measurable complex function on $\C$, then the operator $J_f^{\mathcal E_T}$ is also denoted as $f(T)$.

Recall that a topological space $S$ is called a Polish space if its topology can be induced by a metric that makes $S$ a separable complete space. A
measurable space $S$ is called a standard Borel space if its measurable structure can be induced by a Polish topology on $S$. A spectral measure $E$ on a
measurable space $S$ is called standard if there is a measurable set $S'\subset S$ such that $E(S\setminus S')=0$ and $S'$, considered as a
measurable subspace of $S$, is a standard Borel space.

Let $\{f_\iota\}_{\iota\in I}$ be a family of maps and $S$ be a set contained in the domains of $f_\iota$ for all $\iota\in I$. The family $\{f_\iota\}_{\iota\in I}$ is said to separate points of $S$ if for any two distinct elements $s_1$ and $s_2$ of $S$, there is
$\iota\in I$ such that $f_\iota(s_1)\neq f_\iota(s_2)$.

\begin{definition}\label{d_exact}
Let $S$ be a measurable space and $E$ be a spectral measure on $S$. A family $\{f_\iota\}_{\iota\in I}$ of maps defined $E$-a.e. on $S$ is said to be $E$-separating on
$S$ if $I$ is countable and $\{f_\iota\}_{\iota\in I}$ separates points of $S\setminus N$ for some $E$-null set $N$.
\end{definition}

Our aim is to prove the next result.

\begin{theorem}\label{t0}
Let $S$ be a measurable space, $\mathfrak H$ be a separable Hilbert space, and $E$ be a standard spectral measure for $(S,\mathfrak H)$.
The algebra $\mathcal A(\mathcal P_E)$ is generated by a set $\mathcal X$ of closed densely defined operators in $\mathfrak H$ if and only if the following
conditions hold
\begin{enumerate}
\item For every $T\in \mathcal X$, there is an $E$-measurable complex function $f$ on $S$ such that $T=J^E_f$.

\item There is an $E$-separating family $\{f_\iota\}_{\iota\in I}$ of $E$-measurable complex functions on $S$ such that $J^E_{f_\iota}\in\mathcal X$
for all $\iota\in I$.

\end{enumerate}
\end{theorem}

\begin{remark}\label{rem1}
Every Abelian von~Neumann algebra $\mathcal M$ acting in a separable Hilbert space $\mathfrak H$ is equal to $\mathcal A(\mathcal P_E)$ for a suitable standard spectral measure $E$. Indeed, by Th\'eor\`eme~2 of Sec.~II.6.2 in~\cite{Dixmier}, there are a finite measure $\nu$ on a compact metrizable
space $S$, a $\nu$-measurable family $\mathfrak S(s)$ of Hilbert spaces, and a unitary operator $V\colon \mathfrak H \to \int^\oplus_S \mathfrak
S(s)\,d\nu(s)$ such that $\mathcal M$ coincides with the set of all operators $V^{-1}\mathcal T_g V$, where $\mathcal T_g$ is the operator of multiplication by a $\nu$-measurable $\nu$-essentially bounded function $g$ on $S$. Then we can define $E$ by setting $E(A)=V^{-1}\mathcal T_{\chi_A} V$ for any Borel set $A\subset S$, where $\chi_A$ is the characteristic function of $A$.
\end{remark}

\begin{remark}
Every standard Borel space can be mapped into the segment $[0,1]$ by a measurable one-to-one function (see Lemma~\ref{l4a} below). Hence, given a standard spectral measure $E$ on a measurable space $S$, there always exists an $E$-measurable real bounded function $f$ on $S$ such that the family containing this single function is $E$-separating. By Theorem~\ref{t0}, it follows that $\mathcal A(\mathcal P_E)$ is generated by the self-adjoint operator $J^E_f$. In view of remark~\ref{rem1}, this implies the well-known fact~\cite{Neumann} that every Abelian von~Neumann algebra acting in a separable Hilbert space is generated by a single self-adjoint element.
\end{remark}

\begin{example}
Let $\nu$ be a positive measure on a measurable space $S$ and let the spectral measure $E$ in the Hilbert space $L^2(S,\nu)$ be given by $E(A)=\mathcal T_{\chi_A}$, where $\mathcal T_{\chi_A}$ is the operator of multiplication by the characteristic function $\chi_A$ of $A$. Then $\mathcal A(\mathcal P_E)$ is identified with the algebra $L^\infty(S,\nu)$ (see Lemma~\ref{l4d} below) acting by multiplication on $L^2(S,\nu)$. If $\nu$ is standard, then it follows from Theorem~\ref{t0} that a set $\mathcal C\subset L^\infty(S,\nu)$ generates $L^\infty(S,\nu)$ if and only if $\mathcal C$ contains a $\nu$-separating family on $S$. This statement can be viewed as an analogue of the Stone-Weierstrass theorem in the setting of $L^\infty$-spaces.
\end{example}

We say that two sets $\mathcal X$ and $\mathcal Y$ of closed densely defined operators in $\mathfrak H$ are equivalent if $\mathcal A(\mathcal X) =
\mathcal A(\mathcal Y)$. We say that $\mathcal X$ is equivalent to a closed densely defined operator $T$ if $\mathcal X$ is equivalent to the
one-point set $\{T\}$.

\begin{example}\label{e2}
Let $\Lambda\subset \C$ be a set having an accumulation point in $\C$ and let $f_\lambda(z)=e^{\lambda z}$ for $\lambda\in\Lambda$ and $z\in\C$.
Clearly, we can choose a countable set $\Lambda_0\subset \Lambda$ that has an accumulation point in $\C$. If $f_\lambda(z)=f_\lambda(z')$ for some
$z,z'\in \C$ and all $\lambda\in\Lambda_0$, then we have $z=z'$ by the uniqueness theorem for analytic functions and, therefore, the family
$\{f_\lambda\}_{\lambda\in\Lambda_0}$ separates points of $\C$. Theorem~\ref{t0} therefore implies that $\mathcal P_E$ is equivalent to
$\{J^E_{f_\lambda}\}_{\lambda\in\Lambda}$ for any spectral measure $E$ on $\C$. If $T$ is a normal operator, then applying Theorem~\ref{t0} to the identity function on $\C$ yields the well-known fact (see, e.g., Theorem~3 of Sec.~6.6 in~\cite{BirmanSolomjak}) that $T$ is equivalent to $\mathcal P_{\mathcal E_T}$. This implies that $T$ is equivalent to the set of all operators $e^{\lambda T}$ with $\lambda\in\Lambda$.
\end{example}

The rest of the paper is devoted to the proof of Theorem~\ref{t0}.

\begin{lemma}\label{l4}
Let $\{\mathcal X_\iota\}_{\iota\in I}$ and $\{\mathcal Y_\iota\}_{\iota\in I}$ be families of sets of closed densely defined operators in $\mathfrak
H$ and let $\mathcal X=\bigcup_{\iota\in I}\mathcal X_\iota$ and $\mathcal Y=\bigcup_{\iota\in I}\mathcal Y_\iota$. If $\mathcal X_\iota$ and
$\mathcal Y_\iota$ are equivalent for every $\iota\in I$, then $\mathcal X$ and $\mathcal Y$ are equivalent.
\end{lemma}
\begin{proof}
Set $\mathcal M_\iota = (\mathcal X_\iota\cup \mathcal X^*_\iota)'$ and $\mathcal M = (\mathcal X\cup \mathcal X^*)'$. Then we have $\mathcal M=\bigcap_{\iota\in I} \mathcal M_\iota$.
It follows that $\mathcal M'=\mathcal A(\mathcal X)$ coincides with the von~Neumann algebra generated by $\bigcup_{\iota\in I}\mathcal M'_\iota=\bigcup_{\iota\in I}\mathcal
A(\mathcal X_\iota)$ (see~\cite{Dixmier}, Sec.~I.1.1, Proposition~1). Analogously, $\mathcal A(\mathcal Y)$ is the von~Neumann algebra generated by
$\bigcup_{\iota\in I}\mathcal A(\mathcal Y_\iota)$. Since $\mathcal A(\mathcal X_\iota)=\mathcal A(\mathcal Y_\iota)$ for all $\iota$, it follows
that $\mathcal A(\mathcal X)=\mathcal A(\mathcal Y)$. The lemma is proved.
\end{proof}

\begin{lemma}\label{l41}
Let $\mathfrak H$ be a separable Hilbert space and $\mathcal X$ be a set of closed densely defined operators in $\mathfrak H$. Then there is a
countable subset $\mathcal X_0$ of $\mathcal X$ which is equivalent to $\mathcal X$.
\end{lemma}
\begin{proof}
We first note that every subset of $L(\mathfrak H)$ is separable in the strong topology. Indeed, for any $M\subset L(\mathfrak H)$, we have $M =
\bigcup_{n=1}^\infty M\cap\mathcal B_n$, where $\mathcal B_n=\{T\in L(\mathfrak H) : \|T\|\leq n\}$ is the ball of radius $n$ in $L(\mathfrak H)$.
Since $\mathfrak H$ is separable, $\mathcal B_n$ endowed with the strong topology is a separable metrizable space for any $n$
(see,~e.g.,~\cite{Dixmier}, Sec.~I.3.1). This implies that $M\cap\mathcal B_n$ is separable for any $n$ and, hence, $M$ is separable in the strong
topology.

Let $\mathfrak A = \bigcup_{\mathcal Y\subset \mathcal X} \mathcal A(\mathcal Y)$, where $\mathcal Y$ runs through all finite subsets of $\mathcal
X$. Obviously, $\mathcal A(T)$ is equivalent to $T$ for any closed densely defined operator $T$ and, therefore, Lemma~\ref{l4} implies that $\mathcal
X$ is equivalent to $\bigcup_{T\in \mathcal X}\mathcal A(T)$. Since the latter set is contained in $\mathfrak A$ and $\mathfrak A\subset \mathcal
A(\mathcal X)$, we conclude that $\mathfrak A$ is equivalent to $\mathcal X$. We now note that $\mathfrak A$ is an involutive subalgebra of
$L(\mathfrak H)$ containing the identity operator and, therefore, is strongly dense in $\mathfrak A''=\mathcal A(\mathcal X)$ (\cite{Dixmier},
Sec.~I.3.4, Lemma~6). Let $\mathfrak R$ be a strongly dense countable subset of $\mathfrak A$. For any $R\in\mathfrak R$, we choose a finite set
$\mathcal Y_R\subset \mathcal X$ such that $R\in \mathcal A(\mathcal Y_R)$ and put $\mathcal X_0=\bigcup _{R\in\mathfrak R}\mathcal Y_R$. Clearly,
$\mathcal X_0$ is a countable set. The algebra $\mathcal A(\mathcal X_0)$ is strongly dense in $\mathcal A(\mathcal X)$ because it contains
$\mathfrak R$. On the other hand, $\mathcal A(\mathcal X_0)$ is a von~Neumann algebra and, therefore, is strongly closed. We hence have $\mathcal
A(\mathcal X_0)=\mathcal A(\mathcal X)$, i.e., $\mathcal X_0$ is equivalent to $\mathcal X$. The lemma is proved.
\end{proof}

Let $E$ be a spectral measure for $(S,\mathfrak H)$. For any $E$-measurable complex function $f$ on $S$, the operator $J^E_f$ is normal, and we have $J^E_{\bar f}=J^{E*}_f$, where $\bar f$ is the complex conjugate function of $f$. For any
$E$-measurable $f$ and $g$ on $S$, we have
\begin{equation}\label{sum_prod}
J^E_{fg}=\overline{J^E_fJ^E_g},\quad  J^E_{f+g}=\overline{J^E_f+J^E_g},
\end{equation}
where the bar means closure (see Theorem~7 of Sec.~5.4 in~\cite{BirmanSolomjak}). The operator $J^E_f$ is everywhere defined and bounded if and only if $f$ is $E$-essentially bounded. In this case, we
have
\begin{equation}\label{normbound}
\|J^E_f\| = E\mbox{-}\mathrm{ess\,sup}_{s\in S} |f(s)|.
\end{equation}

Let $S$ and $S'$ be measurable spaces, $\mathfrak H$ be a Hilbert space, $E$ be a spectral measure for $(S,\mathfrak H)$, and $\varphi\colon S\to S'$
be an $E$-measurable map. We denote by $\varphi_* E$ the push-forward of $E$ under $\varphi$. By definition (see~\cite{BirmanSolomjak}, Sec.~5.4), this means that $\varphi_* E$ is the
spectral measure for $(S',\mathfrak H)$ such that
\[
\varphi_* E(A) = E(\varphi^{-1}(A))
\]
for any measurable $A\subset S'$. If $f$ is a $\varphi_*E$-measurable complex function on $S'$, then $f\circ \varphi$ is an $E$-measurable function
on $S$, and we have
\begin{equation}\label{pushforward}
J^{\varphi_*E}_f = J^E_{f\circ \varphi}.
\end{equation}
Let $\varphi$ be an $E$-measurable complex function on $S$. Formula~(\ref{pushforward}) with $f=\mathrm{id}_{\C}$ yields $J^{\varphi_*
E}_{\mathrm{id}_{\C}}=J^E_\varphi$. In view of the uniqueness of $\mathcal E_{J^E_\varphi}$, this means that
\begin{equation}\label{600}
\mathcal E_{J^E_\varphi}=\varphi_* E.
\end{equation}
By~(\ref{pushforward}), it follows that $f\circ\varphi$ is $E$-measurable and
\begin{equation}\label{composition}
f(J^E_\varphi) = J^{\varphi_* E}_f = J^E_{f\circ \varphi}
\end{equation}
for any $\mathcal E_{J^E_\varphi}$-measurable complex function $f$ on $\C$.

It follows from Theorem~5.6.18 in~\cite{KR1} that a closed densely defined operator $T$ in $\mathfrak H$ is normal if and only if the algebra $\mathcal A(T)$ is Abelian.

The next lemma implies, in particular, that condition~(1) of Theorem~\ref{t0} holds if and only if every element of $\mathcal X$ is affiliated with $\mathcal A(\mathcal P_E)$.

\begin{lemma}\label{l4d}
Let $S$ be a measurable space, $\mathfrak H$ be a separable Hilbert space, and $E$ be a spectral measure for $(\mathfrak H,S)$. Then $\mathcal
A(\mathcal P_E)$ coincides with the set of all $J^E_f$, where $f$ is an $E$-measurable $E$-essentially bounded complex function on $S$. A closed
densely defined operator $T$ in $\mathfrak H$ is equal to $J^E_f$ for an $E$-measurable complex function $f$ on $S$ if and only if
\begin{equation}\label{at}
\mathcal A(T)\subset \mathcal A(\mathcal P_E)
\end{equation}
\end{lemma}
\begin{proof}
By Theorem~8 of Sec.~5.4 in~\cite{BirmanSolomjak}, we have $\mathcal A(J^E_f)\subset \mathcal A(\mathcal P_E)$ for any $E$-measurable complex
function $f$ on $S$. This implies that $J^E_f\in \mathcal A(\mathcal P_E)$ for $E$-essentially bounded $f$ because $J^E_f$ belongs to $L(\mathfrak
H)$ for such $f$ and, hence, is contained in $\mathcal A(J^E_f)$. Conversely, Theorem~5 of Sec.~7.4 in~\cite{BirmanSolomjak} shows that any element
of $\mathcal A(\mathcal P_E)$ is equal to $J^E_f$ for some $E$-measurable $E$-essentially bounded complex function $f$ on $S$. It remains to prove
that any closed densely defined operator $T$ such that (\ref{at}) holds is equal to $J^E_f$ for some $E$-measurable complex function $f$ on $S$.
Since $\mathcal A(\mathcal P_E)$ is Abelian, it follows from~(\ref{at}) that $\mathcal A(T)$ is Abelian and, therefore, $T$ is
normal. Let $\chi$ be a complex function on $\C$ defined by the relation
\[
\chi(z) = \frac{z}{|z|+1}.
\]
It is easy to see that the function $\chi$ is one-to-one and maps $\C$ onto the open unit disc $\mathbb D = \{z\in\C : |z|<1\}$. Its inverse function
$\chi^{-1}$ is given by
\[
\chi^{-1}(z) = \frac{z}{1-|z|},\quad |z|<1.
\]
Since $\chi$ is bounded and measurable on $\C$, we have $\chi(T)\in\mathcal A(\mathcal P_{\mathcal E_T})$ and, hence,
$\chi(T)\in\mathcal A(T)$ because $\mathcal A(\mathcal P_{\mathcal E_T})=\mathcal A(T)$ (see Example~\ref{e2}). In view of~(\ref{at}), this implies that $\chi(T)\in \mathcal A(\mathcal P_E)$ and, therefore, $\chi(T)=J^E_g$ for some
$E$-measurable function $g$ on $S$. As $\C\setminus \mathbb D$ is a $\chi_*\mathcal E_T$-null set and $\chi^{-1}$ is a measurable map from $\mathbb
D$ to $\C$, the function $\chi^{-1}$ is $\chi_*\mathcal E_T$-measurable on $\C$. By~(\ref{600}), we have $\mathcal E_{\chi(T)}=\chi_* \mathcal E_T$
and, hence, $\chi^{-1}$ is $\mathcal E_{\chi(T)}$-measurable. It therefore follows from~(\ref{composition}) that
\[
T = \chi^{-1}(\chi(T)) = \chi^{-1}(J^E_g) = J^E_f
\]
for $f=\chi^{-1}\circ g$. The lemma is proved.
\end{proof}

Given a topological space $S$, we denote by $C(S)$ the space of all continuous
complex functions on $S$.

\begin{lemma}\label{l01}
Let $S$ be a Polish space, $\mathfrak H$ be a Hilbert space, and $E$ be a spectral measure for $(S,\mathfrak H)$. Let $\mathcal C$ be a subset of
$C(S)$ that separates the points of $S$ and $\mathcal X$ be the set of all operators $J^E_f$ with $f\in\mathcal C$. Then $\mathcal A(\mathcal
X)=\mathcal A(\mathcal P_E)$.
\end{lemma}

In the proof below, all spectral integrals are taken with respect to $E$, and we write for brevity $J_f$ instead of $J^E_f$.
\begin{proof}
Since $\mathcal A(\mathcal X)\subset \mathcal A(\mathcal P_E)$ by Lemma~\ref{l4d}, we have to show that $\mathcal A(\mathcal P_E)\subset \mathcal
A(\mathcal X)$.

Let $\bar{\mathcal C}$ denote the set of functions, complex conjugate to the elements of $\mathcal C$, and let $\mathfrak A$ be the subalgebra of
$C(S)$ generated by $\mathcal C\cup \bar{\mathcal C}$ and the constant functions. Fix $R\in (\mathcal X\cup \mathcal X^*)'$ and let $\mathfrak A_R$
denote the subset of $C(S)$ consisting of all $f$ such that $J_f$ commutes with $R$. If $f,g\in \mathfrak A_R$, then both $J_fJ_g$ and $J_f+J_g$
commute with $R$. It follows from~(\ref{sum_prod}) that both $J_{fg}$ and $J_{f+g}$ commute\footnote{If a closable densely defined operator $T$ commutes with $R\in L(\mathfrak H)$, then its closure $\bar T$ also commutes with $R$ because $R=(R^*)^*$ and $\bar T = (T^*)^*$.} with $R$, i.e.,
$fg\in\mathfrak A_R$ and $f+g\in\mathfrak A_R$. Hence, $\mathfrak A_R$ is an algebra. Since $\mathfrak A_R$ obviously contains $\mathcal C\cup
\bar{\mathcal C}$ and all constant functions, we have $\mathfrak A_R\supset \mathfrak A$. Thus, every element of $(\mathcal X\cup \mathcal X^*)'$
commutes with any operator $J_f$ with $f\in\mathfrak A$.

Given $f\in C(S)$ and a compact set $K\subset S$, we set $B_{f,K}=J_f E(K)$. Let $\Psi\in \mathfrak H$ and $\Phi=E(K)\Psi$. Then for any measurable
set $A$, we have $E_\Phi(A)=E_\Psi(A\cap K)$ and, therefore, $E_\Phi$ is a finite measure supported by $K$. In view of~(\ref{dom}), this implies that
$\Phi\in D_{J_f}$, i.e., the range of $E(K)$ is contained in the domain of $J_f$. Since $E(K)=J_{\chi_K}$, where $\chi_K$ is the characteristic
function of $K$, it follows from~(\ref{sum_prod}) that $B_{f,K}=J_{f\chi_K}$. Hence, $B_{f,K}\in L(\mathfrak H)$ and~(\ref{normbound}) implies that
\begin{equation}\label{norm}
\|B_{f,K}\|\leq \sup_{s\in K}|f(s)|.
\end{equation}
If $f,g\in C(S)$, then~(\ref{sum_prod}) implies that $B_{f+g,K}=B_{f,K}+B_{g,K}$.

We now show that
\begin{equation}\label{incl}
(\mathcal X\cup \mathcal X^*)'\subset {\mathcal Y}',
\end{equation}
where $\mathcal Y$ is the set of all $J_f$ with $f\in C(S)$. Let $R\in (\mathcal X\cup \mathcal X^*)'$. We first prove that $\mathcal Y'$ contains
all operators $R_K=E(K)RE(K)$, where $K$ is a compact subset of $S$. Fix $f\in C(S)$ and let $\varepsilon>0$. Since $\mathcal C\subset \mathfrak A$,
the algebra $\mathfrak A$ separates points of $S$, and the Stone--Weierstrass theorem implies that there is $g\in \mathfrak A$ such that
$|f(s)-g(s)|< \varepsilon$ for any $s\in K$. Since $R_K$ commutes with both $J_g$ and $E(K)$, it follows that $R_K$ commutes with $B_{g,K}$. In view
of~(\ref{norm}), we have
\begin{multline}\nonumber
\|B_{f,K}R_K-R_K B_{f,K}\|\leq \|B_{f,K}R_K-B_{g,K}R_K\|+ \|R_K B_{g,K}-R_K B_{f,K}\|\leq \\ \leq 2\|B_{f-g,K}\|\|R\|<2\varepsilon\|R\|.
\end{multline}
Because $\varepsilon$ is arbitrary, this means that $R_K$ commutes with $B_{f,K}$. This implies that $R_K$ commutes with $J_f$ because
$B_{f,K}R_K=J_f R_K$ and $R_KB_{f,K}$ is an extension of $R_K J_f$ by the commutativity of $E(K)$ and $J_f$. This proves that $R_K\in \mathcal Y'$.
By Lemma~\ref{l2}, $\mathcal Y'$ is strongly closed. Hence, inclusion~(\ref{incl}) will be proved if we
demonstrate that every strong neighborhood of $R$ contains $R_K$ for some compact set $K$. To this end, it suffices to show that for every
$\Psi\in\mathfrak H$ and $\varepsilon>0$, there is a compact set $K_{\Psi,\varepsilon}$ such that
\begin{equation}\label{est}
\|(R-R_K)\Psi\|\leq \varepsilon
\end{equation}
for any compact set $K\supset K_{\Psi,\varepsilon}$. Since
\[
R-R_K = E(K)RE(S\setminus K) + E(S\setminus K)R,
\]
we have
\[
\|(R-R_K)\Psi\|\leq \|R\|\, E_\Psi(S\setminus K)^{\frac{1}{2}} + E_\Phi(S\setminus K)^{\frac{1}{2}},
\]
where $\Phi = R\Psi$. As $S$ is a Polish space, Theorem~1.3 in~\cite{Billingsley} ensures that there is a compact set $K_{\Psi,\varepsilon}$ such
that both $\|R\|^2 E_\Psi(S\setminus K_{\Psi,\varepsilon})$ and $E_\Phi(S\setminus K_{\Psi,\varepsilon})$ do not exceed $\varepsilon^2/4$ and, therefore,
(\ref{est}) holds for any $K\supset K_{\Psi,\varepsilon}$. Inclusion~(\ref{incl}) is thus proved.

We next show that
\begin{equation}\label{incl1}
\mathcal Y'\subset \mathcal P'_E.
\end{equation}
Let $R\in\mathcal Y'$. For any closed set $F\subset S$, it is easy to construct a uniformly bounded sequence of functions $f_n\in C(S)$ that
converges pointwise to $\chi_F$. Then $J_{f_n}$ strongly converge to $J_{\chi_F}=E(F)$ (see Theorem~2 of Sec.~5.3 in~\cite{BirmanSolomjak}). Since
$J_{f_n}$ commute with $R$ for all $n$, this implies that $E(F)$ commutes with $R$. Let $\Sigma^R$ denote the set of all measurable sets $A\subset S$
such that $E(A)$ commutes with $R$. We have proved that $\Sigma^R$ contains all closed sets. If $A\in\Sigma^R$, then $E(S\setminus A)=1-E(A)$
commutes with $R$ and, hence, $S\setminus A\in\Sigma^R$. If $A_1,A_2\in\Sigma^R$, then both $E(A_1\cap A_2)=E(A_1)E(A_2)$ and $E(A_1\cup A_2)=
E(A_1)+E(A_2)-E(A_1)E(A_2)$ commute with $R$ and, therefore, $A_1\cap A_2$ and $A_1\cup A_2$ belong to $\Sigma^R$. Let $A_n$ be a sequence of
elements of $\Sigma^R$ and $A=\bigcup_{n=1}^\infty A_n$. For all $n=1,2,\ldots$, we set $B_n=\bigcup_{j=1}^n A_j$. Then $B_n\in\Sigma^R$ for all $n$,
and the $\sigma$-additivity of $E$ implies that $E(B_n)$ converge strongly to $E(A)$. Hence, $E(A)$ commutes with $R$, i.e., $A\in\Sigma^R$. We thus
see that $\Sigma^R$ is a $\sigma$-algebra containing all closed sets. This implies that $\Sigma^R$ coincides with the Borel $\sigma$-algebra, and
(\ref{incl1}) is proved.

Inclusions~(\ref{incl}) and~(\ref{incl1}) imply that $(\mathcal X\cup \mathcal X^*)'\subset \mathcal P'_E$ and, hence, $\mathcal A(\mathcal
P_E)\subset \mathcal A(\mathcal X)$. The lemma is proved.
\end{proof}

The next lemma summarizes the facts about Polish and standard Borel spaces that are needed for the proof of Theorem~\ref{t0}.

\begin{lemma}\label{l4a}
\begin{enumerate}
\item[1.] Let $S$ and $S'$ be standard Borel spaces and $f\colon S\to S'$ be a one-to-one measurable mapping. Then $f(S)$ is a measurable subset of
$S'$ and $f$ is a measurable isomorphism from $S$ onto $f(S)$.

\item[2.] Let $S$ be a Polish space and $B$ be its Borel subset. Then there are a Polish space $P$ and a continuous one-to-one map $g\colon P\to S$
such that $B=g(P)$.

\item[3.] If $S$ is a standard Borel space, then there exists a one-to-one measurable function from $S$ to the segment $[0,1]$.

\item[4.] Every measurable subset of a standard Borel space is itself a standard Borel space.
\end{enumerate}
\end{lemma}
\begin{proof}
Statement~1 follows from Theorem~3.2 in~\cite{Mackey}, which, in its turn, is a reformulation of a theorem by Souslin (see~\cite{Kuratowski},
Chapter~III, Sec.~35.IV). For the proof of statement~2, see Lemma~6 of Sec.~IX.6.7 in~\cite{Bourbaki}. To prove statement~3, we recall that every
standard Borel space is either countable or isomorphic to the segment $[0,1]$ (see~\cite{Takesaki}, Appendix, Corollary~A.11). In the latter case,
any isomorphism between $S$ and $[0,1]$ gives us the required function. If $S$ is countable, then we can just choose any one-to-one map from $S$ to
$[0,1]$ because all functions on $S$ are measurable. Since every standard Borel space can be endowed with a Polish topology inducing its measurable structure, statement~4 follows from statements~1 and~2. The lemma is proved.
\end{proof}

\begin{lemma}\label{l4b}
Let $E$ be a spectral measure on a standard Borel space $S$, $I$ be a countable set and $\{f_\iota\}_{\iota\in I}$ be a family of measurable
complex-valued functions on $S$ that separates the points of $S$. Then the von~Neumann algebra generated by all operators $J_{f_\iota}$ with
$\iota\in I$ coincides with $\mathcal A(\mathcal P_E)$.
\end{lemma}
\begin{proof}
Let $f$ denote the map $s\to \{f_\iota(s)\}_{\iota\in I}$ from $S$ to $\C^I$. The space $\C^I$ endowed with its natural product topology is a Polish
space, and the measurability of $f_\iota$ implies that of $f$. Since $f_\iota$ separate the points of $S$, the map $f$ is one-to-one. By statement~1
of Lemma~\ref{l4a}, $f(S)$ is a Borel subset of $\C^I$ and $f$ is a measurable isomorphism of $S$ onto $f(S)$. By statement~2 of Lemma~\ref{l4a},
there are a Polish space $P$ and a continuous one-to-one map $g\colon P\to \C^I$ such that $f(S)=g(P)$. Hence, $h = f^{-1}\circ g$ is a measurable
one-to-one map from $P$ onto $S$. By statement~1 of Lemma~\ref{l4a}, $h$ is a measurable isomorphism from $P$ onto $S$. We now use $h$ to transfer
the topology from $P$ to $S$, i.e., we say that a set $O\subset S$ is open if and only if $h^{-1}(O)$ is open in $P$. Once $S$ is equipped with this
topology, $h$ becomes a homeomorphism between $P$ and $S$ and, hence, $S$ becomes a Polish space. Since $h$ is a measurable isomorphism, the Borel
measurable structure generated by the topology of $S$ coincides with its initial measurable structure. Because $f=g\circ h^{-1}$ is continuous, all
$f_\iota$ are continuous. Hence, the statement follows from Lemma~\ref{l01}. The lemma is proved.
\end{proof}

\begin{proof}[Proof of Theorem~$\ref{t0}$] {$\,$}
\par\noindent

Suppose conditions~(1) and~(2) hold. By Lemma~\ref{l4d}, condition~(1) implies that $\mathcal A(\mathcal X)\subset \mathcal A(\mathcal P_E)$. Let the
family $\{f_\iota\}_{\iota\in I}$ be as in condition~(2) and $\mathcal X_0$ be the set of all $J^E_{f_\iota}$ with $\iota\in I$. Without loss of generality, we can assume that all $f_\iota$ are everywhere defined and measurable on $S$. Using statement~4 of Lemma~\ref{l4a} and the fact that $E$ is
standard, we can choose a measurable subset $\tilde S$ of $S$ such that $E(S\setminus\tilde S)=0$, the family $\{f_\iota\}_{\iota\in I}$ separates points of $\tilde S$, and $\tilde S$, considered as a measurable subspace of
$S$, is a standard Borel space. Let $\tilde E$ denote the restriction of $E$ to $\tilde S$. For each $\iota\in I$, let
$\tilde f_\iota$ be the restriction of $f_\iota$ to $\tilde S$. Then we have $J^{\tilde E}_{\tilde
f_\iota}=J^E_{f_\iota}$ for all $\iota\in I$ and it follows from Lemma~\ref{l4b} that $\mathcal A(\mathcal X_0)=\mathcal A(\mathcal P_{\tilde E})$.
As $\mathcal P_E = \mathcal P_{\tilde E}$, this implies that $\mathcal A(\mathcal P_E)\subset \mathcal A(\mathcal X)$ and, hence, $\mathcal
A(\mathcal P_E) = \mathcal A(\mathcal X)$.

Conversely, let $\mathcal A(\mathcal P_E) = \mathcal A(\mathcal X)$. Then condition~(1) is ensured by Lemma~\ref{l4d}. By Lemma~\ref{l41}, there is a
countable set $\mathcal X_0\subset \mathcal X$ such that $\mathcal A(\mathcal X_0)=\mathcal A(\mathcal X)$. Choose a countable family
$\{f_\iota\}_{\iota\in I}$ of measurable complex functions on $S$ such that $\mathcal X_0$ coincides with the set of all operators $J^E_{f_\iota}$.
It suffices to show that $\{f_\iota\}_{\iota\in I}$ is $E$-separating. Let $f$ be the measurable map $s\to \{f_\iota(s)\}_{\iota\in I_0}$ from
$S$ to $\C^I$. For each $\iota\in I$, let $\pi_\iota\colon \C^{I}\to\C$ be the function taking $\{z_\kappa\}_{\kappa\in I}$ to $z_\iota$. For any
$\iota\in I$, we have $\pi_\iota\circ f = f_\iota$, and it follows from~(\ref{pushforward}) that $J^{f_*E}_{\pi_\iota} = J^E_{f_\iota}$ for all
$\iota\in I$, i.e., the set of all $J^{f_*E}_{\pi_\iota}$ coincides with $\mathcal X_0$. Since the family $\{\pi_\iota\}_{\iota\in I}$ obviously
separates the points of $\C^{I}$, Lemma~\ref{l01} implies that $\mathcal A(\mathcal X_0) = \mathcal A(\mathcal P_{f_*E})$ and, hence,
\begin{equation}\label{spec_mes}
\mathcal A(\mathcal P_{E}) = \mathcal A(\mathcal P_{f_*E}).
\end{equation}
As above, let $\tilde S$ be a measurable subset of $S$ such that $E(S\setminus\tilde S)=0$ and $\tilde S$, considered as a measurable subspace of
$S$, is a standard Borel space. By statement~3 of Lemma~\ref{l4a}, there exists a one-to-one measurable function $g$ on $\tilde S$. Clearly, $g$ is $E$-measurable on $S$ and it
follows from Lemma~\ref{l4d} and~(\ref{spec_mes}) that $\mathcal A(J^E_g)\subset \mathcal A(\mathcal P_{f_*E})$. Now Lemma~\ref{l4d} implies that there
exists a measurable function $h$ on $\C^{I}$ such that $J^E_g = J^{f_*E}_h$. In view of~(\ref{pushforward}), this means that $J^E_g = J^E_{h\circ f}$
and, hence, $g$ and $h\circ f$ are equal $E$-a.e.\footnote{We have $J^E_f=0$ for an $E$-measurable $f$ if and only if $f=0$ $E$-a.e. (see~Lemma~2 of of Sec.~5.4 in~\cite{BirmanSolomjak}). By~(\ref{sum_prod}), this implies that $J^E_f=J^E_g$ for $E$-measurable $f$ and $g$ if and only if $f=g$ $E$-a.e.} Since $g$ is one-to-one on $\tilde S$, it follows that $f$ is one-to-one on $\tilde S\setminus N$,
where $N$ is the $E$-null set of all $s\in\tilde S$ such that $g(s)\neq h(f(s))$. This means that $\{f_\iota\}_{\iota\in I}$ is $E$-separating and the theorem
is proved.
\end{proof}

\section*{Acknowledgments}
The author is grateful to I.V.~Tyutin and B.L.~Voronov for useful discussions. This research was supported by the Russian Foundation for Basic Research (Grant No. 09-01-00835)


\begin{thebibliography}{1}

\bibitem{BirmanSolomjak}
M.~Sh. Birman and M.~Z. Solomjak.
\newblock {\em Spectral theory of selfadjoint operators in {H}ilbert space}.
\newblock Mathematics and its Applications (Soviet Series). D. Reidel
  Publishing Co., Dordrecht, 1987.
\newblock Translated from the 1980 Russian original by S. Khrushch{\"e}v and V.
  Peller.

\bibitem{GSVT} D.M.~Gitman, A.G.~Smirnov, I.V.~Tyutin and B.L.~Voronov, ``Symmetry preserving self-adjoint extensions of Schrodinger operators with singular potentials,''
arXiv:1012.2588 [math-ph].

\bibitem{Dixmier}
Jacques Dixmier.
\newblock {\em Les alg\`ebres d'op\'erateurs dans l'espace hilbertien
  (alg\`ebres de von {N}eumann)}.
\newblock Gauthier-Villars \'Editeur, Paris, 1969.
\newblock Deuxi{\`e}me {\'e}dition, revue et augment{\'e}e, Cahiers
  Scientifiques, Fasc. XXV.

\bibitem{KR1}
Richard~V. Kadison and John~R. Ringrose.
\newblock {\em Fundamentals of the theory of operator algebras. {I}.
  {E}lementary theory}.
\newblock Academic Press, New York, 1983.

\bibitem{Neumann}
J.~von~Neumann.
\newblock Zur Algebra der Funktionaloperationen und Theorie der normalen Operatoren.
\newblock {\em Math. Ann.}, 102:370–-427, 1930.

\bibitem{Billingsley}
Patrick Billingsley.
\newblock {\em Convergence of probability measures}.
\newblock Wiley Series in Probability and Statistics: Probability and
  Statistics. John Wiley \& Sons Inc., New York, second edition, 1999.
\newblock A Wiley-Interscience Publication.

\bibitem{Mackey}
George~W. Mackey.
\newblock Borel structure in groups and their duals.
\newblock {\em Trans. Amer. Math. Soc.}, 85:134--165, 1957.

\bibitem{Kuratowski}
Casimir Kuratowski.
\newblock {\em Topologie. {I}. {E}spaces {M}\'etrisables, {E}spaces
  {C}omplets}.
\newblock Monografie Matematyczne, vol. 20. Warszawa-Wroc\l aw, 1948.
\newblock 2d ed.

\bibitem{Bourbaki}
N.~Bourbaki.
\newblock {\em \'{E}l\'ements de math\'ematique. {T}opologie g\'en\'erale.
  {C}hapitres 5 \`a 10}.
\newblock Hermann, Paris, 1974.

\bibitem{Takesaki}
Masamichi Takesaki.
\newblock {\em Theory of operator algebras. {I}}.
\newblock Springer-Verlag, New York, 1979.

\end{thebibliography}
\end{document}